\documentclass[12pt]{amsart}
\usepackage{amsfonts,amsmath,amssymb,ifthen,latexsym,calc,cite}
\usepackage[a4paper]{geometry}
\usepackage[colorlinks]{hyperref}
\usepackage[capitalize]{cleveref}

\newtheorem{thm}{Theorem}[section]
\newtheorem{cor}[thm]{Corollary}
\newtheorem{lem}[thm]{Lemma}
\newtheorem{prop}[thm]{Proposition}
\theoremstyle{definition}

\newtheorem{conj}[thm]{Conjecture}

\theoremstyle{remark}

\numberwithin{equation}{section}
\numberwithin{figure}{section}
\numberwithin{table}{section}

\def\<{\;<\;}
\def\={\;=\;}
\def\>{\;>\;}

\def\rmand{\quad\text{ and }\quad}
\def\sg{\mathfrak{S}}

\crefname{ineq}{Ineq.}{Ineqs.}
\Crefname{ineq}{Inequality}{Inequalities}
\creflabelformat{ineq}{~\upshape(#2#1#3)}
\crefname{rec}{Recurrence}{Recurrences}
\Crefname{rec}{Recurrence}{Recurrences}
\creflabelformat{rec}{~\upshape(#2#1#3)}
\crefname{rl}{Relation}{Relations}
\crefname{rl}{Relation}{Relations}
\creflabelformat{rl}{~\upshape(#2#1#3)}
\crefname{fml}{Formula}{Formulas}
\crefname{fml}{Formula}{Formulas}
\creflabelformat{fml}{~\upshape(#2#1#3)}
\crefname{rst}{Result}{Results}
\crefname{rst}{Result}{Results}
\creflabelformat{rst}{~\upshape(#2#1#3)}
\crefrangeformat{equation}{Eqs.~(#3#1#4) ---~(#5#2#6)}

\crefname{def}{Def.}{Defs.}
\Crefname{def}{Def.}{Defs.}
\creflabelformat{def}{~\upshape(#2#1#3)}
\crefname{thm}{Theorem}{Theorems}
\Crefname{thm}{Theorem}{Theorems}
\creflabelformat{thm}{~\upshape #2#1#3}
\crefname{prop}{Proposition}{Propositions}
\Crefname{prop}{Proposition}{Propositions}
\creflabelformat{prop}{~\upshape #2#1#3}
\crefname{conj}{Conjecture}{Conjectures}
\Crefname{conj}{Conjecture}{Conjectures}
\creflabelformat{conj}{~\upshape #2#1#3}
\crefname{eg}{Example}{Examples}
\Crefname{eg}{Example}{Examples}
\creflabelformat{eg}{~\upshape #2#1#3 }

\newcounter{hours}
\newcounter{minutes}
\newcommand{\printtime}{
\setcounter{hours}{\time/60}
\setcounter{minutes}{\time-\value{hours}*60}
\ifthenelse{\value{hours}<10}{0}{}\thehours:
\ifthenelse{\value{minutes}<10}{0}{}\theminutes}
\usepackage{xcolor}

\newcommand{\mdef}[1]{\textit{\textbf{#1}}}
\newcommand{\vsb}{\vskip-6pt}
\newcommand{\Stir}[2]{\genfrac{[}{]}{0pt}{}{#1}{#2}}

\begin{document}

\title{The CLLC conjecture holds for cyclic outer permutations}

\author[J.L. Gross]{Jonathan L. Gross}
\address{Department of Computer Science, Columbia University, New York, 10027 NY, USA}
\email{gross@cs.columbia.edu}

\author[T. Mansour]{Toufik Mansour}
\address{Department of Mathematics, University of Haifa, 31905 Haifa, Israel}
\email{tmansour@univ.haifa.ac.il}

\author{Thomas W. Tucker}
\address{Department of Mathematics, Colgate University, Hamilton, NY 13346, USA}
\email{ttucker@colgate.edu}

\author[D.G.L. Wang]{David G.L. Wang$^{\dag\ddag}$}
\address{
$^\dag$School of Mathematics and Statistics, Beijing Institute of Technology, 102488 Beijing, P. R. China\\
$^\ddag$Beijing Key Laboratory on MCAACI, Beijing Institute of Technology, 102488 Beijing, P. R. China}
\email{glw@bit.edu.cn}

\keywords{cyclic permutation, Hultman number, log-concave sequence, the Hermite--Biehler theorem}

\subjclass[2010]{05C10, 05A20}

\date{}

\begin{abstract}
Recently, Gross et al.\ posed the LLC conjecture for the locally log-concavity of the genus distribution of every graph, and provided an equivalent combinatorial version, the CLLC conjecture, on the log-concavity of the generating function counting cycles of some permutation compositions. In this paper, we confirm the CLLC conjecture for cyclic permutations, with the aid of Hultman numbers and by applying the Hermite--Biehler theorem on the generating function of Stirling numbers of the first kind. This leads to a further conjecture that every local genus polynomial is real-rooted.
\end{abstract}


\maketitle
\tableofcontents


The LCGD conjecture of Gross et al.~\cite{GRT89} in 1989
is that the genus polynomial of every graph is log-concave.  This has been proved for a substantial number of interesting families of graphs.  However, finding either a counterexample or a proof of the general conjecture has turned out during the past decades to be quite hard.

Most of the proofs of log-concavity for special families of graphs are based on partitioning of the genus polynomial according to the incidence of face-boundary walks on the root-vertices.
Recently, Gross et al.~\cite{GMTW15X} designed a new two-step strategy toward a proof of the general LCGD conjecture based on a new form of partition of the genus polynomial into the set of local genus polynomials at a single root-vertex.

The first step is to prove the log-concavity of the \textit{local genus polynomials} of every rooted graph $(G,v)$.  The genus polynomial of the graph $G$ is the sum of the local genus polynomials at $v$.  This first step has been transformed into the purely combinatorial \textit{CLLC conjecture}, which we describe in the next section.  The second step is to prove the existence of an ordering of local genus polynomials subject to the \textit{synchronicity relation} that was introduced in~\cite{GMTW15}.  Completion of the above two steps together would confirm the LCGD conjecture.

The exposition here is intended to be accessible both to topological graph theorists and to combinatorialists.  For additional contextual information regarding log-concavity and genus distributions, see~\cite{GMTW15}.  General background information regarding topological graph theory can be found in the books~\cite{BW09B,GT01B}.  For surveys of log-concavity see~\cite{Sta89,Bre89,Bre94}.

\bigskip
\section{\large Local Genus Polynomials and IO-Polynomials}

This paper is devoted to extending the classes of cases for which the CLLC Conjecture is affirmed, that is, for which the local genus polynomials are known to be LC.  We introduce the following notations.
\begin{description}
\item[$\sg_n$] the symmetric group on $n$ objects.
\item[$c(\pi)$]  for $\pi\in\sg_n$, the number of cycles of $\pi$.
\item[$Q_n$] the set of $n$-cycles of $\sg_n$.
\end{description}

Given an imbedding of a graph $(G,v)$,  by an \mdef{inner strand} at $v$, we mean a segment of an fb-walk the contains the edge by which the strand enters vertex $v$ and the edge by which that strand leaves $v$.  We may regard the inner strand as permuting the entry edge to the exit edge.  Thus, the imbedding induces a permutation of the edges incident at $v$, which we call the \mdef{inner permutation} \hbox{at $v$}, and which we recognize as equivalent to the rotation at $v$.

By an \mdef{outer strand} at $v$, we mean any component of the complement of the union of the set of inner strands in the union of the set of fb-walks that are incident at $v$.  Each outer strand may be regarded as permuting the edge that precedes its own first edge, in an oriented traversal of the fb-walk containing that strand, to the edge that follows its own last edge.  Thus, the imbedding also induces a second permutation of the edges incident at $v$, which we call the \mdef{outer permutation} \hbox{at $v$}.

We observe that the number of cycles in the product of the inner permutation with the outer permutation for any imbedding is the number of faces that are incident at $v$, which we call \mdef{inner faces}.   An \mdef{outer face} is a face that is not incident \hbox{at $v$}.  Via the Euler formula, the total number of faces of an imbedding determines the genus of the surface.

An assignment $\sigma$ of a rotation to every vertex of $(G,v)$ except $v$ is called an almost complete rotation system, abbreviated ACR.  By a \mdef{local genus polynomial} at $v$, we mean the genus polynomial $\Gamma_{(G,\sigma)}(z)$ for the set $I_\sigma$ of imbeddings that are consistent with $\sigma$.
\smallskip

\begin{conj}[\mdef{Local Log-Concavity Conjecture}, abbreviated as the \mdef{LLC Conjecture}]  \label{conj:LLC}
All the local genus polynomials of every rooted graph $(G,v)$ are log-concave.
\end{conj}

For each outer permutation $\pi$, we define the \mdef{IO-polynomial}
\begin{equation}\label[def]{def:F}
F_\pi(z)=\sum_{\zeta\in Q_n}z^{\lfloor (c(\zeta\pi)-1)/2\rfloor},
\end{equation}
where permutation products are multiplied from left to right.  IO stands for ``inner-outer''.

\begin{thm} \label{thm:power of z}
Let $\pi$ be the outer permutation induced by an ACR $\sigma$ for the rooted graph $(G,v)$.  Then the local genus polynomial at $v$ can be obtained by multiplying the IO-polynomial $F_\pi(z)$ by a power of $z$ that corresponds to the number of outer faces induced by $\sigma$.
\end{thm}

\begin{proof} \vskip-6pt
The number of inner faces for an imbedding is always equal to the product of the corresponding inner permutation and outer permutation.  The number of outer faces for an imbedding equals the difference between the total number of faces and the number of inner faces.  It does not change as the rotation at the root changes, since such a change has no effect on any outer face.  The net effect is that the local genus polynomial is a shift of the IO-polynomial.
\end{proof}

When~$\lambda=n^{k_n}(n-1)^{k_{n-1}}\cdots 1^{k_1}$ is a partition of the integer $n$, with $k_j$ parts of size~$j$ (with parts $j^{0}$ omitted),  we may say that $\lambda$ \mdef{partitions the number} $n = \sum_{j=1}^njk_j=n$ and write $\lambda\vdash n$.

The permutation~$\pi$ is said to be of \mdef{type}~$\lambda$
if the multi-set of cycle lengths of~$\pi$ is~$\lambda$.
As shown by Gross et al.~\cite{GMTW15X},
the polynomial~$F_\pi(z)$ depends only on the type of~$\pi$.
As a consequence, we can define $F_\lambda(z)$ by
\[
F_\lambda(z)\=F_\pi(z),
\]
where $\pi$ is any permutation of type $\lambda$.
We call $F_\lambda(z)$ the \mdef{IO-polynomial with respect to the partition}~$\lambda$.  When the partition of the number $n$ has only one part, we may write $F_{n\vdash n}$ to emphasize that we are regarding the subscript $n$ as a partition.  \smallskip

\begin{conj}[The CLLC conjecture]\label{conj:CLLC}
The IO-polynomial $F_\lambda(z)$ is log-concave for every partition type $\lambda$.
In other words, the polynomial~$F_\pi(z)$ is log-concave for all permutations~$\pi$.
\end{conj}

In view of Theorem \ref{thm:power of z} we see that Conjecture \ref{conj:CLLC} implies Conjecture \ref{conj:LLC}.  Moreover, it is proved in \cite{GMTW15X} that these two conjectures are equivalent.

\Cref{conj:CLLC} has been confirmed \cite{GMTW15X} for the following partition types:
\begin{equation}\label{par}
2^{n/2}\text{ (when $n$ is even)},\quad 21^{n-2},\quad 31^{n-3},\quad 41^{n-4},\rmand 2^21^{n-4},
\end{equation}
which correspond, respectively, to full involutions, transpositions, $3$-cycles, $4$-cycles,
and permutations consisting of two disjoint transpositions.

In this paper, we prove Conjecture 4.1 and Conjecture 4.2 of \cite{GMTW15X}.  That is we will prove that for every partition $\lambda$ of type $r1^{n-r}$ or type $2^s1^{n-2s}$, the polynomial $F_\lambda$ is log-concave.

\bigskip
\section{\large IO-Polynomials of Involutions are LC}\label{sec:reduction}

We prove in this section that when the outer polynomial is an involution, allowing fixed points, the IO-polynomial is log-concave.

\begin{prop}\label{prop:fp}
For any partition $\mu\vdash(n-1)$, we have $F_{\mu1}(z)\=(n-1)F_\mu(z)$.
\end{prop}

\begin{proof} \vsb
Let $\sigma\in\sg_{n-1}$,
and let $\sigma'=\sigma\circ(n)$ be the permutation obtained by adding the letter~$n$ to the permutation~$\sigma$ as a fixed point.
Let $\zeta=(\zeta_1\,\zeta_2\,\cdots\,\zeta_{n-1})\in Q_{n-1}$, and let
\[
\zeta'=(\zeta_1\,\zeta_2\,\cdots\,\zeta_j\,n\,\zeta_{j+1}\,\cdots\,\zeta_{n-1})\in Q_n,
\]
where $j\in[n-1]$.
We claim that
\begin{equation}\label{c=c}
c(\zeta'\sigma')\=c(\zeta\sigma).
\end{equation}

To see this, let
$$C \= \left(q\,\ \zeta\sigma(q)\,\ (\zeta\sigma)^2(q)\,\ \ldots\,\ (\zeta\sigma)^{k-1}(q)\right)$$
be a cycle of the permutation~$\zeta\sigma$.
First suppose that $C$ does not contain the letter~$\zeta_j$.  In that case,
\begin{eqnarray*}
\left(q\,\ \zeta'\sigma'(q)\,\ (\zeta'\sigma')^2(q)\,\ \ldots\,\ (\zeta'\sigma')^{k-1}(q)\right) &=& \\
\left(q\,\ \zeta\sigma(q)\,\ (\zeta\sigma)^2(q)\,\ \ldots\,\ (\zeta\sigma)^{k-1}(q)\right) &=& C
\end{eqnarray*}
is a cycle of the permutation~$\zeta'\sigma'$.  We might say of this that the cycle $C$ is preserved by the transformation $(\zeta\to\zeta', \sigma\to\sigma')$.

Alternatively, we next suppose that the cycle $C$ does contain the letter~$\zeta_j$, say $(\zeta\sigma)^r(z)=\zeta_j$, so that
$$C \= \left(q\,\ \zeta\sigma(q)\,\  \ldots\,\ (\zeta\sigma)^r(q)=\zeta_j\,\ (\zeta\sigma)^{r+1}(q) \,\ \ldots\,\ (\zeta\sigma)^{k-1}(q)\right)$$
In this case, we observe that the cycle $C$ is transformed to
\begin{eqnarray*}
&& \left(q\,\ \zeta'\sigma'(q)\,\ (\zeta'\sigma')^2(q)\,\ \ldots\,\right) \\
&=& \left(q\,\ \zeta\sigma(q)\,\  \ldots\,\ (\zeta\sigma)^r(q)=\zeta_j\,\ \ldots\, \right) \\
&=& \left(q\,\ \zeta\sigma(q)\,\  \ldots\,\ (\zeta\sigma)^r(q)=\zeta_j\,\ (\zeta'\sigma')(\zeta_j)\,\ (\zeta'\sigma')^2(\zeta_j)\,\ \ldots\, \right) \\
&=& \left(q\,\ \zeta\sigma(q)\,\  \ldots\,\ (\zeta\sigma)^r(q)=\zeta_j\,\ \sigma'(n)\!=\!n\,\ (\zeta'\sigma')(n)\,\ \ldots\, \right) \\
&=& \left(q\,\ \zeta\sigma(q)\,\  \ldots\,\ (\zeta\sigma)^r(q)=\zeta_j\,\ n\,\ \sigma'(\zeta_{j+1})\!=\! (\zeta\sigma)(\zeta_j)\,\ \ldots\, \right) \\
&=& \left(q\,\ \zeta\sigma(q)\,\  \ldots\,\ (\zeta\sigma)^r(q)=\zeta_j\,\ n\,\ (\zeta\sigma)^{r+1}(q)\,\  (\zeta\sigma)^{r+2}(q)\,\ \ldots\, \right)
\end{eqnarray*}
which is the cycle obtained by inserting the letter~$n$
immediately after the letter~$\zeta_j$ in the cycle~$C$.
Hence, the permutations~$\zeta'\sigma'$ and~$\zeta\sigma$ have the same number of cycles.
This proves the claimed \cref{c=c}.

For any cyclic permutation $\zeta=(\zeta_1\,\zeta_2\,\cdots\,\zeta_{n-1})\in Q_{n-1}$, we define
\[
S(\zeta)\=\{(\zeta_1\,\zeta_2\,\cdots\,\zeta_j\,n\,\zeta_{j+1}\,\cdots\,\zeta_{n-1})\in Q_n\;\colon\;j\in[n-1]\}.
\]
It is clear that
\[
|S(\zeta)|=n-1
\rmand
Q_n=\sqcup_{\zeta\in Q_{n-1}}S(\zeta).
\]
We also define $c'(\eta)=\lfloor(c(\eta)-1)/2\rfloor$, for any permutation $\eta$.
By using \cref{c=c}, we can infer that
\begin{eqnarray}
F_{\sigma'}(z) &=& \sum_{\zeta'\in Q_n}z^{c'(\zeta'\sigma')}
\=\sum_{\zeta\in Q_{n-1}}\sum_{\zeta'\in S(\zeta)}z^{c'(\zeta'\sigma')}
\=\sum_{\zeta\in Q_{n-1}}\sum_{\zeta'\in S(\zeta)}z^{c'(\zeta\sigma)}\\
&=& \sum_{\zeta\in Q_{n-1}}|S(\zeta)|z^{c'(\zeta\sigma)}
\=(n-1)F_\sigma(z).\label{1fp} \notag
\end{eqnarray}
Since the local genus polynomial depends only on the type of the indicating permutation,
we have completed the proof.
\end{proof}
\smallskip

An immediate corollary is as follows,
which leads us to focus on the partitions with no parts $1$.

\begin{cor}\label{cor:fp}
For any partition $\mu\vdash(n-k)$, we have
\[
F_{\mu1^k}=(n-1)(n-2)\cdots(n-k)F_\mu(z).
\]
\end{cor}

\begin{proof} \vsb
By iterating the application of \cref{prop:fp}.
\end{proof}

As a consequence, all the partitions in \eqref{par} for which \cref{conj:CLLC} has been confirmed, except~$2^{n/2}$, can be handled by considering the partitions of length at most~$4$. The truth for the partition $2^{n/2}\vdash n$ implies \cref{thm:invl} immediately.
\smallskip

\begin{thm}\label{thm:invl}
The local genus polynomial $F_{2^k1^{n-2k}}(z)$ is log-concave for all $k$.
In other words, the CLLC \cref{conj:CLLC} holds for all involutions.
\end{thm}

\bigskip
\section{\large Using Hultman Numbers}\label{sec:floor}

Essentially, \cref{conj:CLLC} concerns the number of cycles in permutations.
We recall that the number permutations in $\sg_n$ with $k$~cycles is given by
the signless Stirling number $\Stir{n}{k}$ of the first kind.  The generating function for Stirling numbers is
\begin{equation}\label{gf:Stirling1}
S_n(z)\=\sum_{\pi\in\sg_n}z^{c(\pi)}\=\sum_{k=1}^n\Stir{n}{k}z^k\=z(z+1)(z+2)\cdots(z+n-1).
\end{equation}
The \mdef{Hultman number}, a closely related number, is the number of permutations in $\sg_n$ whose cycle graph has $k$ alternating cycles (see Hultman~\cite{Hul99T}).

\begin{equation}\label[def]{def:HN}
H(n,k)\=\begin{cases}
\displaystyle \binom{n+2}{2}^{-1}\cdotp \Stir{n}{k},&\text{if $(n-k)$ is odd}\\[8pt]
\qquad 0,&\text{otherwise}.
\end{cases}
\end{equation}
With the aid of Hultman numbers $H(n,k)$, we are able to express the local genus polynomial $F_{n\vdash n}(x)$ in terms of the Stirling numbers $\Stir{n}{k}$.  Then we prove the real-rootedness of the local genus polynomials by applying the Hermite--Biehler theorem, given here as \cref{thm:HB}.
Since, by a theorem of Newton (see~\cite{Bre89} for instance), every real-rooted polynomial is log-concave,
this affirms \cref{conj:CLLC} for cyclic permutations.

The floor function in \cref{def:F} of the polynomial $F_\pi(z)$ was introduced as a technical normalization to satisfy the equivalence of the CLLC conjecture and the topological graph theoretical LLC conjecture; see \cite{GMTW15X}.  In this section, we are going to eliminate the floor function, for another technical purpose, to help deal with the normalized Hultman polynomials.  We now define the polynomial
\begin{equation}\label[def]{def:G}
G_\pi(z)\=\sum_{\zeta\in Q_n}z^{c(\zeta\pi)}.
\end{equation}
We shall express the local genus polynomial $F_\pi(z)$ in terms of the polynomial~$G_\pi(z)$; see \cref{thm:F2G}.

Let $\mu\in \sg_n$.  We call $\mu$ an \mdef{odd permutation} if it can be expressed as the composition of an odd number of transpositions.   Otherwise we call~$\mu$ an \mdef{even permutation}.  We define the \mdef{parity indicator function} for every permutation~$\mu$ as
\begin{equation}\label{def:p}
p(\mu)\=\begin{cases}
1,&\text{if $\mu$ is odd};\\[5pt]
0,&\text{if $\mu$ is even}.
\end{cases}
\end{equation}
Immediately from this definition, we have
\begin{equation}\label{parity:p=p+p}
p(\mu_1\mu_2)\;\equiv\; p(\mu_1)+p(\mu_2)\pmod2
\end{equation}
for any $\mu_1,\mu_2\in \sg_n$.  We denote by~$o(\mu)$ the number of odd-length cycles of the permutation~$\mu$, and by~$e(\mu)$ the number of even-length cycles.  From these definitions, we have
\begin{align}
c(\mu)&\=o(\mu)+e(\mu),
\rmand\label{c=o+e}\\[5pt]
n&\;\equiv\; o(\mu) \pmod2.\label{parity:o=n}
\end{align}
Note that every cycle of odd (resp., even) length is the product of an even (resp., odd) number of transpositions.  Thus we have
\begin{equation}\label{parity:p=e}
p(\mu)\;\equiv\; e(\mu)\pmod2.
\end{equation}
In particular, for $\zeta\in Q_n$, we have
\begin{equation}\label{pf:parity:p=n+1}
p(\zeta)\;\equiv\; n+1\pmod{2}.
\end{equation}
In what follows, we use $\rho_n$ to denote the cyclic permutation $(12\cdots n)$ that permutes each letter~$i$ to~$(i+1)$ modulo~$n$.
\smallskip

\begin{thm}\label{thm:F2G}
Let $\pi\in \sg_n$ and $\zeta\in Q_n$.  Then the parity of the cycle count $c(\zeta\pi)$ differs from that of the  parity indicator $p(\pi)$.  As a consequence, we have
\begin{equation}\label{F2G}
F_{\pi}(z)\=\frac{G_\pi(\sqrt{z})}{(\sqrt{z}\,)^{p(\pi)+1}}.
\end{equation}
\end{thm}

\begin{proof}
We begin with the definition of the polynomial $F_\pi(z)$.
\begin{equation}\label{pf2}
F_{\pi}(z)
\=\sum_{\zeta\in Q_n}z^{\lfloor(c(\zeta\pi)-1)/2\rfloor}
\=\sum_{\zeta\in Q_n\atop{c(\zeta\pi) \text{ is odd}}}z^{(c(\zeta\pi)-1)/2}
+\sum_{\zeta\in Q_n\atop{c(\zeta\pi) \text{ is even}}}z^{c(\zeta\pi)/2-1}.
\end{equation}
If the permutation $\pi$ is even, then the integer $c(\zeta\pi)$ is odd, and thus
\begin{equation}\label{F:even}
F_{\pi}(z)
\=\sum_{\zeta\in Q_n}z^{(c(\zeta\pi)-1)/2}
\=\frac{G_\pi(\sqrt{z})}{\sqrt{z}}.
\end{equation}
Otherwise the permutation $\pi$ is odd, then the integer $c(\zeta\pi)$ is even, and therefore
\begin{equation}\label{F:odd}
F_{\pi}(z)
\=\sum_{\zeta\in Q_n}z^{c(\zeta\pi)/2-1}
\=\frac{G_\pi(\sqrt{z})}{z}.
\end{equation}
From \cref{c=o+e,parity:o=n,def:p,parity:p=p+p,parity:p=e,pf:parity:p=n+1}, we deduce that
\[
c(\zeta\pi)
\=o(\zeta\pi)+e(\zeta\pi)
\,\equiv\, n+p(\zeta\pi) \,\equiv\, n+p(\pi)+p(\zeta) \,\equiv\, p(\pi)+1\!\!\!\pmod2
\]
as desired. In particular,
from \cref{pf:parity:p=n+1}, we infer that
\[
c(\zeta\rho_n)\;\equiv\; p(\rho_n)+1\;\equiv\; n\!\!\!\pmod2.
\]
In view of \cref{def:p,F:even,F:odd}, we obtain the desired \cref{F2G}.
\end{proof}

Since the local genus polynomial $F_\pi(z)$ depends only on the cycle type of the permutation~$\pi$,
from \cref{F2G}, we infer that the polynomial $G_\pi(z)$ also depends only on the cycle type of~$\pi$.
Thus we can define $G_{\lambda}(z)=G_\pi(z)$ for permutations~$\pi$ of type~$\lambda$.
For convenience to deal with the partitions $n\vdash n$ in \cref{sec:cyclic}, we write
\begin{equation}\label[def]{def:FnGn}
F_n(z)\=F_{n\vdash n}(z)
\rmand
G_n(z)\=G_{n\vdash n}(z).
\end{equation}

\bigskip
\section{\large CLLC for Cyclic Permutations}\label{sec:cyclic}

Here, as in Bogart's book~\cite{Bog90B}, a \mdef{cyclic permutation} has at most one cycle of length larger than $1$.  In this section, we confirm \cref{conj:CLLC} (the CLLC Conjecture) for cyclic permutations.

The following interpretation of the Hultman numbers $H(n,k)$ was presented, though implicitly, by Doignon and Labarre~\cite{DL07}.   It was stated explicitly later by B\'ona and Flynn~\cite[Corollary~1]{BF09}.
\begin{lem}\label{lem:BF}
The Hultman number~$H(n,k)$ is equal to the number of $(n+1)$-cycles $\zeta\in Q_{n+1}$
such that the permutation $\rho_{n+1}\zeta$ has exactly~$k$ cycles.
\end{lem}

By \cref{lem:BF}, we can express the local genus polynomial $F_n(z)$ in terms of the Stirling numbers of the first kind.

\begin{thm}\label{thm:F2S}
We have
\[
\binom{n+1}{2}F_n(z)\=\sum_{1\le k\le n+1\atop{k\equiv n\!\!\!\!\pmod2}}\!\!\!\Stir{n+1}{k} z^{\lfloor (k-1)/2\rfloor}.
\]
\end{thm}

\begin{proof}
Note that any permutation and its inverse have the same type of cycles.
By \cref{lem:BF}, it follows that the Hultman number $H(n,k)$ counts $(n+1)$-cycles~$\zeta$ such that
\[
c(\zeta^{-1}\rho_{n+1}^{-1})=k.
\]
Since the map $\zeta\mapsto\zeta^{-1}$ is an involution in the set~$Q_{n+1}$,
the Hultman number $H(n,k)$ counts $(n+1)$-cycles~$\zeta$ such that
\[
c(\zeta\rho_{n+1}^{-1})=k.
\]
Thus, from \cref{def:FnGn} and \cref{def:G}, we derive
\[
G_{n+1}(z)
\=\sum_{\zeta\in Q_{n+1}}z^{c(\zeta\rho_{n+1}^{-1})}
\=\sum_{k=1}^{n+1}H(n,k)z^k.
\]
By \cref{thm:F2G} and \cref{def:HN} of Hultman numbers, we find that
\begin{align}
F_n(z)
&\=\frac{G_{n}(\sqrt{z})}{(\sqrt{z}\,)^{p(\rho_n)+1}}
\=\frac{1}{(\sqrt{z}\,)^{p(\rho_n)+1}}\sum_{k=1}^{n}H(n-1,\,k)z^{k/2}  \label{pf1}\\
&\=\frac{1}{(\sqrt{z}\,)^{p(\rho_n)+1}\binom{n+1}{2}}\sum_{1\le k\le n+1\atop{k\equiv n\!\!\!\!\pmod2}}\Stir{n+1}{k}z^{k/2}. \notag
\end{align}
By \cref{pf:parity:p=n+1}, we have
\begin{equation}\label{p:rho}
p(\rho_n)=\begin{cases}
1,&\text{if $n$ is even};\\[5pt]
0,&\text{if $n$ is odd}.
\end{cases}
\end{equation}
The remaining proof is routine, using \cref{pf1,p:rho}.
\end{proof}
\smallskip

We remark that Stanley~\cite[Exercise 91 (c)]{Sta14} has claimed the formula
\begin{equation}\label{eq:Stanley}
\binom{n+1}{2}G_n(z)\=\sum_{i=0}^{\lfloor (n-1)/2\rfloor}c(n+1,\,n-2i)z^{n-2i},
\end{equation}
without proofs. It is easy to prove \cref{thm:F2S} by using \cref{thm:F2G,eq:Stanley}.
\smallskip

To confirm the CLLC Conjecture for cyclic permutations, we will use the Hermite--Biehler \cref{thm:HB}.  Let us recall some necessary definitions and notation to state it.
\begin{itemize}
\item A polynomial~$P(z)$ is said to be \mdef{Hurwitz stable} (resp., \mdef{weakly Hurwitz stable}) if all roots of~$P(z)$ have negative (resp., nonpositive) real parts.
\item\vskip4pt Let~$f(z)$ and~$g(z)$ be real-rooted polynomials with roots $f_1,f_2,\ldots,f_s$ and  $g_1,g_2,\ldots,g_t$, respectively.  We say that $g(z)$ \mdef{interlaces} $f(z)$, denoted $g(z)\preceq f(z)$, if
either
\[
s=t
\rmand
g_1\le f_1\le g_2\le f_2\le \cdots\le g_t\le f_s,
\]
or
\[
s=t+1
\rmand
f_1\le g_1\le f_2\le g_2\le \cdots\le g_t\le f_s.
\]
\item\vskip4pt If all the above inequalities are strict,
then we say that the polynomial~$g(z)$ \mdef{strictly interlaces}~$f(z)$,
denoted $g(z)\prec f(z)$.
\item\vskip4pt For any polynomial $P(z)=\sum_{k=0}^na_kz^k$, we define
\[
P^e(z)=\sum_{k=0}^{\lfloor{n/2}\rfloor}a_{2k}z^k
\rmand
P^o(z)=\sum_{k=0}^{\lfloor{(n-1)/2}\rfloor}a_{2k+1}z^k,
\]
and call them the \mdef{even part} and the \mdef{odd part} of the polynomial $P(z)$, respectively.
\end{itemize}

\begin{thm}[The Hermite--Biehler theorem]\label{thm:HB}
Let $P(z)$ be a polynomial with real coefficients.
Suppose that the polynomial $P^e(z)P^o(z)$ does not vanish identically.
Then the polynomial~$P(z)$ is Hurwitz stable (resp., weakly Hurwitz stable)
if and only if the polynomials~$P^e(z)$ and~$P^o(z)$
have only real and negative (resp., nonpositive) zeros,
and $P^e(z)\prec P^o(z)$ (resp., $P^e(z)\preceq P^o(z)$).
\end{thm}

\begin{proof}\vsb
This theorem is presented and proved by  Br\"and\'en~\cite[Theorem 4.1]{Bra11}.
\end{proof}

Now we are ready to demonstrate the main result of this paper.

\begin{thm}\label{thm:RR:cyclic}
The local genus polynomial $F_n(z)$ is real-rooted. Consequently,
the CLLC \cref{conj:CLLC} holds true for cyclic permutations.
\end{thm}

\begin{proof}\vsb
From \cref{gf:Stirling1}, we see that the roots of the generating function
\[
S_n(z)=\sum_{k=1}^n\Stir{n}{k}z^k
\]
of the Stirling numbers are $0$, $-1$, $-2$, $\ldots$, $1-n$.  Thus, the function~$S_n(z)$ is weakly Hurwitz stable.  Note that the polynomial
\[
\sum_{1\le k\le n+1\atop{k\equiv n\!\!\!\!\pmod2}}\Stir{n+1}{k}z^{\lfloor (k-1)/2\rfloor}
\]
is either the even part or the odd part of the polynomial $S_{n+1}(z)$,
which is weakly Hurwitz stable.
By \cref{thm:HB,thm:F2S}, both the even part and odd part of the polynomial $S_{n+1}(z)$ have only real and negative zeros.  Therefore, the polynomial above has only real zeros.
By Theorem 4.2, the polynomial $F_n(z)$ is real-rooted.
\end{proof}

\medskip

We now sketch an alternative proof of \cref{thm:RR:cyclic}.
The following recurrence
for the generating function $H_g(x)=\sum_{n=2g}^\infty h_g(n) x^n$ of the Hultman numbers is provided by Alexeev and Zograf~\cite[Theorem 3]{AZ14}:
$$
(n+2)h_g(n)=(2n+1)h_g(n-1) - (n-1)h_g(n-2) + n^2(n-1)h_{g-1}(n-2)
$$
where $h_g(n) = H(n,n+1-2g)$.
From this, one can easily prove that the local genus polynomials~$F_n(z)$ satisfy the recurrence
\begin{equation}\label[rec]{rec:F}
(n+2)F_{n}(z)\=(2n+1)z^{e(n)}F_{{n-1}}(z)+(n-1)(n^2-z)F_{{n-2}}(z),
\end{equation}
with $F_{0}(z)=F_{1}(z)=1$, where the function
\[
e(n)
\=\begin{cases}
1,&\text{if $n$ is even}\\[5pt]
0,&\text{if $n$ is odd}
\end{cases}
\]
indicates the evenness of the integer $n$.
Then, by invoking the following criterion implied immediately by~\cite[Corollary 2.4]{LW07} of Liu and Wang, one may deduce the real-rootedness of the polynomial $F_n(z)$ directly.

\begin{thm}
Let $\{P_n(x)\}$ be a sequence of standard polynomials and satisfy the recurrence
relation
$$P_n(x) = a_n(x)P_{n-1}(x) + b_n(x)P'_{n-1}(x) + c_n(x)P_{n-2}(x),$$
where the three coefficients $a_n(x)$, $b_n(x)$, and $c_n(x)$ are real polynomials such that \hbox{$\deg P_n(x) = \deg P_{n-1}(x)$} or $\deg P_n(x) = \deg P_{n-1}(x) + 1$.  Suppose that for each $n$, the coefficients of $P_n(x$) are nonnegative.  If $b_n(x) \le 0$ and $c_n(x) \le 0$ whenever $x \le 0$, then $\{P_n(x)\}$ forms a generalized Sturm sequence.  In particular, if for each index $n$, deg $P_n(x) = n$ and either $b_n(x) < 0$ or $c_n(x) < 0$ whenever $x \le 0$, then $\{P_n(x)\}$ forms a Sturm sequence.  \qed
\end{thm}

\bigskip
\section{\large Conclusions}

Stahl \cite{Sta97} proved that the genus polynomials for several sequences of graphs are real-rooted genus polynomials, and he proposed a conjecture stronger than the LCGD conjecture, namely, that genus polynomials of all graphs are real-rooted.  Beyond Stahl's examples, several other sequences, including the iterated claws~\cite{GMTW14X}, have been proved to have real-rooted genus polynomials.   However, Chen and Liu~\cite{CL10} have provided a counterexample to Stahl's conjecture.  Interestingly, the counterexample graphs, i.e., which are some chains of $4$-wheels, have been shown~\cite{GMTW15} recently to have log-concave genus polynomials.

Based on brute force computation for partitions of integers up to~$11$,
we offer the following strengthening of the CLLC Conjecture.

\begin{conj}\label{conj:RR}
The polynomial $F_\lambda(x)$ is real-rooted for all partitions~$\lambda$.  In other words, the IO-polynomial $F_\pi(x)$ is real-rooted for all permutations~$\pi$.
\end{conj}

Hultman numbers were used in an enumerative analysis of the cycle graph structure with alternating cycles, which was introduced by Bafna and Pevzner~\cite{BP98} in their pursuit of a shortest sequence of transpositions that transforms a given permutation into the identity permutation.  This graph structure was further studied by various other authors, including \cite{DL07,BF09}.   It has proved to be a powerful tool in evaluating the distance between two genomes in computational biology (see~\cite{AZ14,FMT14}).

\bigskip


\begin{thebibliography}{99}

\bibitem{AZ14}
N. Alexeev and P. Zograf,
Random matrix approach to the distribution of genomic distance,
\textsl{J. Comput. Biol.} \textbf{21}(8) (2014), 622--631.

\bibitem{BP98}
V. Bafna and P. A. Pevzner,
Sorting by transpositions,
\textsl{SIAM J. Discrete Math.}  \textbf{11} (1998), 224--240.

\bibitem{BW09B}
L.W. Beineke, R.J. Wilson, J.L. Gross, and T.W. Tucker, editors,
\textsl{Topics in Topological Graph Theory}, Cambridge Univ.\ Press, 2009.

\bibitem{Bog90B}
K.P. Bogart,
\textsl{Introductory Combinatorics}, 2nd ed.,
Harcourt, Brace, Jovanovich, 1990, pp. 486.

\bibitem{BF09}
M. B\'ona and R. Flynn,
The average number of block interchanges needed to sort a permutation and a recent result of Stanley,
\textsl{Inform. Process. Lett.}  \textbf{109}(16) (2009), 927--931.

\bibitem{Bra11}
P. Br\"and\'en,
Iterated sequences and the geometry of zeros,
\textsl{J. Reine Angew. Math.}  \textbf{658} (2011), 115--131.

\bibitem{Bre89}
F. Brenti, \textsl{Unimodal, Log-concave and P\'olya Frequency Sequences in Combinatorics},
Mem. Amer. Math. Soc. \textbf{413}(81) (1989).

\bibitem{Bre94}
F. Brenti, Log-concave and unimodal sequences in algebra, combinatorics, and geometry: An update,
\textsl{Contemp. Math.} \textbf{178} (1994), 71--89.

\bibitem{CL10}
Y.C. Chen and Y. Liu,
On a conjecture of Stahl,
\textsl{Canad. J. Math.} \textbf{62} (2010), 1058--1059.

\bibitem{DL07}
J.-P. Doignon and A. Labarre,
On Hultman numbers,
\textsl{J. Integer Seq.} \textbf{10} (2007), Article 07.6.2.

\bibitem{FMT14}
P. Feij\~ao, F.V. Martinez, A. Th\'evenin,
On the multichromosomal Hultman number,
in  \textsl{Advances in Bioinformatics and Computational Biology},
Springer (2014), 9--16.

\bibitem{GMTW14X}
J.L.~Gross, T. Mansour, T.W.~Tucker, and D.G.L. Wang,
Iterated claws have real-rooted genus polynomials, \textsl{Ars Math. Contemporanea}, to appear. arXiv: 1501.06105.

\bibitem{GMTW15}
J.L. Gross, T. Mansour, T.W. Tucker, and D.G.L. Wang,
Log-concavity of combinations of sequences and applications to genus distributions,
\textsl{SIAM J. Discrete Math.} \textbf{29}(2) (2015), 1002--1029.

\bibitem{GMTW15X}
J.L. Gross, T. Mansour, T.W. Tucker, and D.G.L. Wang,
Combinatorial conjectures that imply local log-concavity of graph genus polynomials,
\textsl{European J. Combin.}, to appear.

\bibitem{GRT89}
J.L. Gross, D.P. Robbins, and T.W. Tucker,
Genus distributions for bouquets of circles,
    \textsl{J. Combin. Theory Ser. B} \textbf{47} (1989), 292--306.

\bibitem{GT01B}
J.L. Gross and T.W. Tucker,
\textsl{Topological Graph Theory}, Dover, 2001 (original ed. Wiley, 1987).

\bibitem{Hul99T}
A. Hultman, Toric Permutations, Master’s thesis, Department of Mathematics,
KTH, Stockholm, Sweden, 1999.

\bibitem{LW07}
L. Liu and Y. Wang,
A unified approach to polynomial sequences with only real zeros,
\textsl{Adv. Appl. Math.} \textbf{38} (2007), 542--560.

\bibitem{Sta97}
S. Stahl, On the zeros of some genus polynomials,
\textsl{Canad. J. Math.} \textbf{49} (1997), 617--640.

\bibitem{Sta89}
R.P. Stanley,
Log-concave and unimodal sequences in algebra, combinatorics, and geometry,
\textsl{Ann. New York Acad. Sci.} \textbf{576} (1989), 500--534.

\bibitem{Sta14}
R.P. Stanley,
Supplementary exercises (without solutions) for Chapter 7 (symmetric functions) of
\textsl{Enumerative Combinatorics, vol. 2}, Cambridge University Press, 1999''.
Internet resource available at
\href{http://math.mit.edu/~rstan/ec/ch7supp.pdf}{\tt math.mit.edu/\char126 rstan/ec/ch7supp.pdf}.
\end{thebibliography}
\end{document}